\documentclass{amsart}
\newtheorem{thm}{Theorem}[section]
\newtheorem{lemma}[thm]{Lemma}

\newtheorem{thmintro}{Theorem}
\newtheorem*{prob}{Problem}
\DeclareMathOperator{\cof}{cof}

\newcommand{\R}{\mathbb R}
\newcommand{\Z}{\mathbb Z}
\newcommand{\T}{\mathbb T}
\DeclareMathOperator{\SL}{SL}
\DeclareMathOperator{\const}{const}

\DeclareMathOperator{\grad}{grad}

\newcommand{\lieder}{\mathcal L}

\title[Toric extremal KRS are K\"ahler-Einstein]{Toric extremal K\"ahler-Ricci solitons are K\"ahler-Einstein}
\author{Simone Calamai \and David Petrecca}
\address{(S.~Calamai) Dip. di Matematica e Informatica ``U. Dini'' - Universit\`a di Firenze \endgraf Viale Morgagni 67A -  Firenze - Italy}
\email{simocala at gmail.com}
\address{(D.~Petrecca) Institut f\"ur Differentialgeometrie - Leibniz Universit\"at Hannover \endgraf Welfengarten 1 -  Hanover - Germany}
\email{petrecca at math.uni-hannover.de}
\begin{document}
\begin{abstract}
In this short note, we prove that a Calabi extremal K\"ahler-Ricci soliton on a compact toric K\"ahler manifold is Einstein. This solves for the class of toric manifolds a general problem stated by the authors that they solved only under some curvature assumptions.
\end{abstract}
\maketitle

\section*{Introduction}
Let $M^{2n}$ be a compact K\"ahler manifold and let $\Omega \in H^{1,1}(M)$ be a K\"ahler class. In the attempt to identify ``special'' representatives of $\Omega$, several notions of ``canonical'' K\"ahler metrics have been introduced. A natural choice are of course K\"ahler-Einstein metrics, generalized by \emph{extremal} metrics and \emph{K\"ahler-Ricci solitons (KRS)}. Extremal metrics are defined to be critical points of the \emph{Calabi functional}
\[
\omega \mapsto \int_M s_\omega^2 \omega^n
\]
that maps the K\"ahler metric $\omega$ to the $L^2$-norm of its scalar curvature. The Euler-Lagrange equation of the Calabi functional is
\begin{equation} \label{ScalHolo}
\grad_\omega(s_\omega) \text{ is holomorphic}.
\end{equation}

K\"ahler-Ricci solitons are K\"ahler metrics that satisfy the relation
\begin{equation}\label{KRSeq}
\rho +c \omega = \lieder_X \omega
\end{equation}
with their Ricci form $\rho$, for some vector field $X$ that is holomorphic and, in the compact case, is the gradient of a smooth function $f \colon M \to \R$. The KRS equation forces $\omega$ to lie in the class $2 \pi c_1(M)$. 

In \cite{krsextr} we addressed the problem whether the same $\omega \in 2 \pi c_1(M)$ can be extremal and a KRS without being Einstein and we proved the following.

\begin{thmintro}[\cite{krsextr}] \label{thm:krsextr}
	A compact extremal KRS with positive holomorphic sectional curvature is K\"ahler-Einstein.
\end{thmintro}

Toric manifolds are compact K\"ahler $2n$-manifolds admitting an effective Hamiltonian action of an $n$-torus $\T$ by K\"ahler automorphism. Although in an algebraic geometric context, Fulton calls them a ``remarkably fertile testing ground for general theories'' and, also from the K\"ahler geometric point of view, their richness of symmetries makes them a large park of examples. 

As compact symplectic manifolds, they are characterized by the image of their moment map, that is a \emph{Delzant polytope}, i.e. a convex polytope $\Delta \subset \R^n$ with certain combinatoric properties. Given a compact symplectic toric manifold with moment image $\Delta$, all possible compatible complex structure are described by a single function, as we explain below.

The $\T$-invariant K\"ahler geometry on a dense subset is well described in the coordinates given by the moment map itself. In these coordinates, the extremal condition \eqref{ScalHolo} has a particularly simple description, see e.g. \cite{abreu}.

Separately, it is known that every toric Fano manifolds admits a KRS, see e.g. \cite{donaldson_toric} and references therein where, in addition, Donaldson explains also the relation between the soliton field $X$ and the Delzant polytope. The existence of extremal metrics in the toric setting is discussed in \cite{abreu}.

The purpose of this note is to prove the following result. 

\begin{thmintro} \label{thmtoricextremalKRS}
	A compact toric Calabi-extremal K\"ahler-Ricci soliton is K\"ahler-Einstein.
\end{thmintro}	

This solves the problem stated in \cite{krsextr} for the class of toric K\"ahler metrics, that can have holomorphic sectional curvature of any sign and so are not included in Theorem \ref{thm:krsextr}. 

The proof of Theorem \ref{thmtoricextremalKRS} is based on the combinatoric properties of Delzant polytopes and the boundary behavior of the Abreu potential. The problem in its full generality remains open.

\begin{prob}Prove that every extremal K\"ahler-Ricci soliton is Einstein or find a counterexample.
\end{prob}

Another class of manifold related to toric K\"ahler manifolds is given by toric bundles, where the existence of KRS has been studied in \cite{PodSpiToric}. It would be interesting to apply the techniques of toric geometry from \cite{abreu, donaldson_toric} to study the existence of extremal or constant scalar curvature K\"ahler metrics in this class of manifolds and establish an analogue of Theorem \ref{thmtoricextremalKRS}.

\subsection*{Acknowledgements}
The first named author is supported by SIR 2014 AnHyC ``Analytic aspects in complex and hypercomplex geometry" (code RBSI14DYEB) and by GNSAGA of INdAM; he also wants to thank Xiuxiong Chen for constant support. The second named author is supported by the  Research Training Group 1463 ``Analysis, Geometry and String Theory'' of the DFG as well as GNSAGA of INdAM. The authors are also grateful to Fabio Podest\`a for his interest in this work and his feedback.

\section{Proof of Theorem \ref{thmtoricextremalKRS}}

Let $(M, g, \omega)$ be a toric K\"ahler manifold, with moment map $\mu \colon M \to \Delta = \mu(M) \subset \R^n$. The moment image can be written as
\begin{equation} \label{delzantpoly}
\Delta = \{ x \in \R^n: \ell_k(x) \geq b_k, 1 \leq k \leq d \}
\end{equation}
as intersection of the $d$ half-spaces $\{ x \in \R^n: \ell_k(x)-b_k \geq 0 \}$.

The linear functions $\ell_k$ are defined by $\ell_k(x) = \langle u_k, x \rangle$, where $v_k$ is the normal to the \emph{facet} $\{ \ell_k(x) = 0 \} \cap \Delta$. The combinatoric property of being Delzant implies the following.

\begin{lemma} \label{lemma:hyperplane}
		Let $\Delta$ be a Delzant polytope in $\R^n$. Then the vertices of $\Delta$ cannot lie on the any affine hyperplane.
\end{lemma}	
\begin{proof}
		Let $P$ be a vertex of $\Delta$. By definition of Delzant polytope, the exactly $n$ edges meeting at $P$ are of the form $t v_i$ for $t \in [0, a_i]$ and the $v_i$ can be taken to be a basis of $\Z^n$. Further $n$ vertices are  of the form $P_i = a_i v_i$ and they cannot lie on the same affine hyperplane of $\R^n$ as the $v_i$ are linearly independent over $\R$. 
\end{proof}	

Given a compact toric symplectic manifold $(M, \omega)$ with Delzant polytope $\Delta$, consider the dense subset
\[
M^0 = \{ p \in M: \text{ the $\T$-action is free at $p$} \} \simeq \Delta^0 \times \T,
\]
where $\Delta^0$ is the interior of $\Delta$ and $(x,y) \in \Delta^0 \times \T$ are the \emph{symplectic coordinates}. In these coordinates, the $\T$-action is just the group multiplication on the second component. In particular, $\T$-invariant tensor fields on $M^0$ depend only on $x \in \Delta^0$.

All $\T$-invariant complex structures compatible with $\omega$ are determined by the \emph{Abreu potential}, a function $g\colon \Delta^0 \to \R$ given by 
	\begin{equation} \label{AbreuPot}
	2g(x) = \sum \ell_k(x) \log \ell_k(x) + h(x),
	\end{equation}
on the interior of $\Delta$, where the $\ell_k$ are from \eqref{delzantpoly} and $h$ is a smooth function on $\Delta$.

In the $(x,y)$-coordinates, the symplectic form is the canonical $\omega = dx_i \wedge dy_i$ and the K\"ahler metric corresponding $g$ as in \eqref{AbreuPot} it is $ g_{,ij}(x) dx_i \cdot dx_j$, where $G = (g_{,ij})$ is the (Euclidean) Hessian of $g$. The matrix $G$ has to be singular on the boundary of $\Delta$ in order for the metric to extend smoothly on the whole $M$. However, it is possible to describe the behavior of $G$ on the vertices of $\Delta$.
	
	\begin{lemma} \label{lemma:vertices}
		The inverse of the Hessian matrix $G$ vanishes at the vertices of $\Delta$.
	\end{lemma}
	
	\begin{proof}
		Without loss of generality, up to translations and to a transformation of $\SL(n,\Z)$, we can assume that $0$ is a vertex and that the edges meeting there are the coordinate axes $x_1, \ldots, x_n$.
		
		The transformed polytope is then given by
		\[
		\Delta = \bigcap_{i=1}^n \{ x \in \R^n: x_i  \geq 0 \} \cap \bigcap_{i=n+1}^d \{ x \in \R^n: \ell_k(x) \geq 0 \}
		\]
		and the linear functions $\ell_k$ do not vanish at zero.
		
	The Abreu potential $g$ is given by
	\begin{equation*}
	2g(x) = \sum_{i=1}^n x_i \log x_i + \underbrace{\sum_{i=n+1}^d	\ell_i(x) \log \ell_i(x) + h(x)}_{=: \tilde h(x)}
	\end{equation*}
		and its Hessian matrix is 
		\begin{equation} \label{eq:G}
		G_{ij} = \frac{\delta_{ij}}{x_j} + \tilde h_{,ij}(x)
		\end{equation}
		where the function $\tilde h_{,ij}$ is given by
		\begin{equation} \label{eq:derivh}
		\tilde h_{,ij} = \sum_{k=n+1}^d \frac{\ell_{k,i}(x) \ell_{k,j}(x)}{\ell_k(x)} + h_{,ij}.
		\end{equation}
		
		From \cite[Thm.~2.8]{abreu},  the determinant of $G$ is given by 
		\begin{equation*}
		\frac 1 {\det G} = \delta(x) x_1 \cdots x_n \cdot \ell_{n+1}(x) \cdots \ell_d(x)
		\end{equation*}
		for some function $\delta$ \emph{strictly positive and smooth on the whole $\Delta$}.
		
		The entry $g^{ij}$ of $G^{-1}$ is given by
		\[
		g^{ij}(x) = \frac 1 {\det G} \cof(G)_{ij}
		\]
		where $\cof(G)$ is the cofactor matrix of $G$.
		
		The conclusion follows from the claim that
		\begin{equation*}
		\cof(G)_{ij} = o \biggl ( \frac 1 {x_1 \cdots x_n} \biggr ).
		\end{equation*}
		
		From \eqref{eq:G} one can see that, after eliminating the $i$-th row and the $j$-th column, the variables $x_i$ and $x_j$ can appear at the denominator only in the derivatives of $\tilde h$, but from \eqref{eq:derivh} we see that their limit for $x \to 0$ is finite, so the claim is true.
	\end{proof}

Abreu's characterization \cite{abreu} of toric extremal metrics relies on the fact that a $\T$-invariant function has a holomorphic gradient if, and only if, it is an affine function in the symplectic coordinates. We use this on the scalar curvature and on the potential $f$ of $X$.

\begin{proof}[Proof of Theorem \ref{thmtoricextremalKRS}]
	From the preserved quantity (in our notation)
	\[
	s + |\nabla f|^2 + 2f = \const
	\]
	that holds for every Ricci soliton, see e.g. \cite{chow}, plus the extremal assumption, it follows that both $f$ and $ |\nabla f|^2$ are affine functions in the interior of $\Delta$.
	
	If $f = a \cdot x$, then one has that
	\[
	 |\nabla f|^2 = a^T G^{-1}(x)a
	\]
	is an affine function  as well. If we consider its extension to the whole $\R^n$, it is zero in all the vertices of $\Delta$ by the Lemma \ref{lemma:vertices}. On the other hand, the zeros of a nonzero affine function is a proper affine hyperplane, so by Lemma \ref{lemma:hyperplane} we can conclude that the length of $X$ must be the zero function. So $X=0$ and the metric is Einstein.
\end{proof}

\bibliography{../../allbib/allbib}
\bibliographystyle{amsplain}

\end{document}